\documentclass[11pt, twoside]{amsart}
\usepackage{amssymb}
\usepackage[all]{xy}

\newcommand{\Sym}{\mathrm{Sym}}
\newcommand{\Alt}{\mathrm{Alt}}
\newcommand{\supp}{\mathrm{supp}}
\newcommand{\fin}{\mathrm{fin}}
\newcommand{\w}{\omega}
\newcommand{\id}{\mathrm{id}}
\newcommand{\IR}{\mathbb R}
\newcommand{\IQ}{\mathbb Q}

\newcommand{\U}{\mathcal U}
\newcommand{\HH}{\mathcal H}
\newcommand{\IN}{\mathbb N}
\newcommand{\IP}{\mathbb P}
\newcommand{\IC}{\mathbb C}
\newcommand{\Tau}{\mathcal T}
\newcommand{\cn}{cn}
\newcommand{\Zeta}{\mathfrak Z}

\newtheorem{theorem}{Theorem}
\newtheorem{lemma}[theorem]{Lemma}

\newtheorem{problem}[theorem]{Problem}
\newtheorem{question}[theorem]{Question}
\newtheorem{proposition}[theorem]{Proposition}
\newtheorem{example}[theorem]{Example}
\theoremstyle{definition}

\begin{document}
\title{Cardinal invariants distinguishing permutation groups}
\author{Taras Banakh and Heike Mildenberger}
\address{Taras Banakh: Jan Kochanowski University in Kielce (Poland) and Ivan Franko National University of Lviv (Ukraine)}
\email{t.o.banakh@gmail.com}

\address{Heike Mildenberger:  Abteilung f\"ur Mathematische Logik,
Mathematisches Institut, Universit\"at Freiburg, Eckerstr.~1,
 79104 Freiburg im Breisgau, Germany}
\email{heike.mildenberger@math.uni-freiburg.de}

\begin{abstract} We prove that for infinite cardinals $\kappa<\lambda$ the alternating group $\Alt(\lambda)$ (of even permutations) of $\lambda$ is not embeddable into the symmetric group $\Sym(\kappa)$ (of all permutations) of $\kappa$. To prove this fact we introduce and study several monotone cardinal group invariants which take value $\kappa$ on the groups  $\Alt(\kappa)$ and $\Sym(\kappa)$.
\end{abstract}

\subjclass{03E15, 20B07, 20B30, 20B35, 54A10, 54A25}

\date{June 22, 2015}

\keywords{Symmetric groups, alternating group, $\kappa^+$-compatibility condition, embeddings, $\sigma$-discrete space, weight, spread, uniformly discrete subset, group topology, linear group topology, cardinal group invariant}

\maketitle

By Cayley's classical theorem \cite[1.6.8]{Rob}, each group $G$ embeds into the group $\Sym(|G|)$ of all bijective transformations of the cardinal $|G|$. Observe that for a symmetric group $G=\Sym(\kappa)$ on an infinite cardinal $\kappa$ Cayley's Theorem can be improved: the group $G=\Sym(\kappa)$ embeds into the symmetric group $\Sym(\log|G|)$. This suggests the following question: {\em can each infinite group $G$ be embeded into the symmetric group $\Sym(\log|G|)$?} Here for a cardinal $\kappa$ by $\log(\kappa)=\min\{\lambda:\kappa\le 2^\lambda\}$ we denote the {\em logarithm} of $\kappa$.
Another question of the same flavor asks: {\em can the symmetric group $\Sym(\kappa)$ on an infinite cardinal $\kappa$ be embedded into the symmetric group $\Sym(\lambda)$ on a smaller cardinal $\lambda<\kappa$?} In this paper we shall give negative answers to both questions.
First, we need to introduce some notation.

Let $\kappa$ be a (finite or infinite) cardinal. By ${\rm Sym}(\kappa)$ we denote the set of
bijective functions from $\kappa$ to $\kappa$,
also called the {\em permutations} of $\kappa$.  The set $\Sym(\kappa)$ endowed with the operation of composition of permutations is a group called the {\em symmetric group} on $\kappa$. This group contains a normal subgroup $\Sym_\fin(\kappa)$ consisting of permutations $f:\kappa\to\kappa$ with finite support $\supp(f)=\{x\in\kappa:f(x)\ne x\}$. A permutation $f:\kappa\to\kappa$ with two-element support $\supp(f)$ is called a {\em transposition} of $\kappa$. It is well-known that each finitely supported permutation can be written as a finite composition of transpositions. A permutation $f:\kappa\to\kappa$ is called {\em even} if it can be written as the composition of an even number of transpositions. The even permutations form a subgroup $\Alt(\kappa)$ of $\Sym(\kappa)$ called the {\em alternating group} on $\kappa$. It is a normal subgroup of index 2 in $\Sym_\fin(\kappa)$. So, we get the inclusions $$\Alt(\kappa)\subset\Sym_\fin(\kappa)\subset\Sym(\kappa).$$
For an infinite cardinal $\kappa$ these groups have cardinalities $|\Alt(\kappa)|=|\Sym_{\fin}(\kappa)|=\kappa$ and $\Sym(\kappa)=2^\kappa$.

The following theorem answers in negative the questions posed at the beginning of the paper.

\begin{theorem}\label{t1}
Let $\kappa <\lambda$ be two infinite cardinals. Then
there is no embedding
of $\Alt(\lambda)$
into $\Sym(\kappa)$.
\end{theorem}

The idea of the proof of this theorem is rather natural: find a cardinal characteristic $\varphi(G)$ of a group $G$ which is invariant under isomorphisms of groups, is monotone under taking subgroups, and takes value $\kappa$ on the groups $\Alt(\kappa)$ and $\Sym(\kappa)$. Then for any infinite cardinals $\kappa<\lambda$ we would have $\varphi(\Alt(\lambda))=\lambda\not\le \kappa=\varphi(\Sym(\kappa))$, and the monotonicity of $\varphi$ would imply that $\Alt(\lambda)$ does not embed into $\Sym(\kappa)$. In the sequel, cardinal characteristics of groups which are invariant under isomorphisms of groups will be called {\em cardinal group invariants}. A cardinal group invariant $\varphi$ is called {\em monotone} if $\varphi(H)\le\varphi(G)$ for any subgroup $H$ of a group $G$.

Many examples of monotone cardinal group invariants can be produced as minimizations of cardinal characteristics of (semi)topological groups over certain families of admissible topologies on a given group. Now we explain this approach in more details. First, we define four families $\Tau_s(G)$, $\Tau_c(G)$, $\Tau_g(G)$, and $\Tau_l(G)$ of admissible topologies on a group $G$.

A topology $\tau$ on a group $G$ is called {\em shift-invariant} if for every $a,b\in G$ the two-sided shift $s_{a,b}:G\to G$, $s_{a,b}:x\mapsto axb$, is a homeomorphism of the topological space $(G,\tau)$.
This is equivalent to saying that the group multiplication $G\times G\to G$, $(x,y)\to xy$, is separately continuous. A group endowed with a shift-invariant topology is called a {\em semitopological group}, (see \cite{AT}). For a group $G$ by $\Tau_s(G)$ we  denote the family of all Hausdorff shift-invariant topologies on $G$.

We shall say that a topology $\tau$ on a group $G$ has {\em separately continuous commutator} if the function $G\times G\to G$, $(x,y)\mapsto xyx^{-1}y^{-1}$, is separately continuous. By $\Tau_c(G)$ we denote the family of Hausdorff shift-invariant topologies on $G$ having separately continuous commutator.

A topology $\tau$ on a group $G$ is called a {\em group topology} if the function $G\times G\to G$, $(x,y)\mapsto xy^{-1}$, is jointly continuous. By $\Tau_g(G)$ we denote the family of Hausdorff group topologies on $G$.

A group topology $\tau$ on a group $G$ is called {\em linear} if it has a neighborhood base at the unit $1_G$ of $G$ consisting of $\tau$-open subgroups of $G$. By $\Tau_l(G)$ we denote the  family of linear Hausdorff group topologies on $G$.

It follows that $\Tau_s(G)\supset\Tau_c(G)\supset\Tau_g(G)\supset\Tau_l(G)$ for every group $G$.

By a {\em cardinal topological invariant of semitopological groups} we understand a function $\varphi$ assigning to each semitopological group $G$ some cardinal $\varphi(G)$ so that $\varphi(G)=\varphi(H)$ for any topologically isomorphic semitopological groups $G,H$. We shall say that the function $\varphi$ is {\em monotone} if $\varphi(H)\le\varphi(G)$ for any subgroup $H$ of a semitopological group $G$.

Any (monotone) cardinal topological invariant $\varphi$ of semitopological groups induces four (monotone) cardinal group invariants $\varphi_s$, $\varphi_c$, $\varphi_g$, $\varphi_l$ assigning to each group $G$ the cardinals
\begin{itemize}
\item $\varphi_s(G)=\min\{\varphi(G,\tau):\tau\in\Tau_s(G)\}$,
\item $\varphi_c(G)=\min\{\varphi(G,\tau):\tau\in\Tau_c(G)\}$,
\item $\varphi_g(G)=\min\{\varphi(G,\tau):\tau\in\Tau_g(G)\}$,
\item $\varphi_l(G)=\min\{\varphi(G,\tau):\tau\in\Tau_l(G)\}$.
\end{itemize}
The inclusions $\Tau_s(G)\supset\Tau_c(G)\supset\Tau_g(G)\supset\Tau_l(G)$ imply the inequalities
$$\varphi_s(G)\le\varphi_c(G)\le \varphi_g(G)\le\varphi_l(G).$$

We shall apply this construction to  three cardinal topological invariants of semitopological groups: the weight $w$, the spread $s$ and the uniform spread $u$. A subset $D$ of a semi-topological group $(G,\tau)$ is called {\em uniformly discrete} if there exists a $\tau$-open neighborhood $U\subset G$ of the unit $1_G$ such that $y\notin xU$ for any distinct points $x,y\in G$.
For a semi-topological group $(G,\tau)$  let
\begin{itemize}
\item $w(G,\tau)=\min\{|\mathcal B|:\mathcal B\subset\tau$ is a base of the topology $\tau\}$ be the {\em weight} of $(G,\tau)$;
\item $s(G,\tau)=\sup\{|D|:D\subset G$ is a discrete subspace of $(G,\tau)\}$ be the {\em spread} of $(G,\tau)$ and
\item $u(G,\tau)=\sup\{|D|:D\subset G$ is a uniformly discrete subset of $(G,\tau)\}$ be the {\em uniform spread} of $(G,\tau)$.
\end{itemize}
Observe that the weight and spread of $(G,\tau)$ depend only on the topology $\tau$ whereas the definition of the uniform spread involves both structures (algebraic and topological) of the semitopological group $(G,\tau)$. Taking into account that each uniformly discrete subset of a semi-topological group $(G,\tau)$ is discrete, we conclude that $u(G)\le s(G)\le w(G)$. Theorem~5.5 of \cite{Hodel} implies that $w(G)\le 2^{2^{s(G)}}$. It is easy to see that $u,s,w$ are monotone cardinal topological invariants of semitopological groups. Minimizing these cardinal functions over the families $\Tau_s(G)$, $\Tau_c(G)$, $\Tau_g(G)$ and $\Tau_l(G)$ we obtain 12 monotone cardinal group invariants that relate as follows. In the diagram an arrow $\varphi\to\psi$ between two cardinal group invariants $\varphi,\psi$ indicates that $\varphi(G)\le\psi(G)$ for any group $G$.
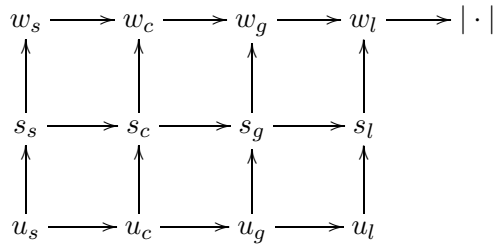
$$
\xymatrix{
w_s\ar[r]&w_c\ar[r]&w_g\ar[r]&w_l\ar[r]&|\cdot|\\
s_s\ar[r]\ar[u]&s_c\ar[r]\ar[u]&s_g\ar[r]\ar[u]&s_l\ar[u]\\
u_s\ar[r]\ar[u]&u_c\ar[r]\ar[u]&u_g\ar[r]\ar[u]&u_l\ar[u]
}
$$

Next, we define a combinatorial cardinal group invariant $\cn(G)$ called the {\em weak compatibility number} of a group $G$. It is defined as the smallest infinite cardinal $\kappa$ for which the group $G$ satisfies the {\em weak $\kappa^+$-compatibility condition}:
\begin{itemize}
\item  for any finite group $F$, and isomorphisms $f_i:F\to F_i$, $i<\kappa$, onto finite subgroups $F_i$ of $G$ there are two indices $i<j<\kappa^+$ and a homomorphism $\phi:\langle F_i\cup F_j\rangle\to F$ such that $\phi\circ f_i=\phi\circ f_j=\id_F$.
\end{itemize}
Here by $\kappa^+$ we denote the successor cardinal of $\kappa$ and for a subset $A\subset G$ by $\langle A\rangle$ we denote the subgroup of $G$ generated by $A$. The weak $\kappa^+$-compatibility condition is a weak version of a notion introduced in \cite[Def. 1.9]{JST}. The definition of the weak compatibility number implies that it is a monotone cardinal group invariant. Observe that each torsion-free group $G$ has $\cn(G)=\w$, so $\cn(G)$ can be much smaller than the cardinal $s_s(G)\ge \log\log|G|$.

In Proposition~\ref{p:lw-char} we shall present an algebraic description of the linear weight $w_l$ and using this description will prove that $\cn(G)\le w_l(G)$ for any group $G$. This inequality allows us to add the weak compatibility number $\cn$ to the diagram describing the relations between cardinal group invariants and obtain the diagram:
$$
\xymatrix{
w_s\ar[r]&w_c\ar[r]&w_g\ar[r]&w_l&\cn.\ar[l]\\
s_s\ar[r]\ar[u]&s_c\ar[r]\ar[u]&s_g\ar[r]\ar[u]&s_l\ar[u]\\
u_s\ar[r]\ar[u]&u_c\ar[r]\ar[u]&u_g\ar[r]\ar[u]&u_l\ar[u]
}
$$

The following theorem implies Theorem~\ref{t1} and can be considered as a main result of this paper.

\begin{theorem}\label{main} For any infinite cardinal $\kappa$ and a group $G$ with $\Alt(\kappa)\subset G\subset \Sym(\kappa)$ we get $$\kappa=s_s(G)=u_c(G)=w_l(G)=\cn(G).$$
\end{theorem}

The proof of this theorem will be divided into three Lemmas~\ref{l1}, \ref{l2}, \ref{l3}. We start our proofs with an algebraic description of the linear weight $w_l(G)$ of a group $G$.

\begin{proposition}\label{p:lw-char} For an infinite group $G$ its linear weight $w_l(G)$ is equal to the smallest cardinal $\kappa$ for which there are subgroups $G_i$, $i\in\kappa$, of index $|G/G_i|\le\kappa$ such that $\bigcap_{i\in\kappa}G_i$ coincides with the trivial subgroup $\{1_G\}$ of $G$.
\end{proposition}

\begin{proof}  Let $w_l'(G)$ denote the smallest cardinal $\kappa$ such that $\{1_G\}=\bigcap_{i\in\kappa}G_i$ for some subgroups $G_i$ of index $|G/G_i|\le\kappa$ in $G$. We need to prove that $w_l(G)=w'_l(G)$.

To prove that $w_l'(G)\le w_l(G)$, use the definition of the linear weight $w_l(G)$ and find a linear group topology $\tau$ of weight $\kappa=w_l(G)$ on $G$. Let $\mathcal B\subset\tau$ be a base of the topology $\tau$ of cardinality $|\mathcal B|=\kappa$. Let $\mathcal B_1=\{B\in\mathcal B:1_G\in B\}$ be the neighborhood base at the unit $1_G$ of the group $G$. Since $|\mathcal B_1|\le|\mathcal B|\le\kappa$, the set $\mathcal B_1$ can be enumerated as $\mathcal B_1=\{B_i\}_{i<\kappa}$. Since the topology is linear, each set $B_i\in\mathcal B_1$ contains an open subgroup $H_i$ of $G$. Taking into account that the family $\{xH_i:x\in G\}$ is disjoint and each coset $xH_i$, $x\in G$, contains some basic set $U\in\mathcal B$, we conclude that $|\{xH_i:x\in G\}|\le|\mathcal B|=\kappa$ and hence the subgroup $H_i$ has index $\le\kappa$ in $G$. The Hausdorff property of the topology $\tau$ guarantees that $\{1_G\}=\bigcap\mathcal B_1=\bigcap_{i<\kappa}H_i$. So, the family $\{H_i\}_{i<\kappa}$ witnesses that $w'_l(G)\le \kappa=w_l(G)$.

Now we check that $w_l(G)\le w_l'(G)$. By the definition of the cardinal $\kappa=w_l'(G)$, there exists a family $\mathcal H$ of subgroups of index $\le\kappa$ in $G$ such that $|\mathcal H|\le\kappa$ and $\{1_G\}=\bigcap\mathcal H$. Observe that for any subgroup $H\in\mathcal H$, any $x\in G$, and any $y\in xH$, we get $xHx^{-1}=yHy^{-1}$, which implies that the family $\{xHx^{-1}:x\in G\}$ has cardinality $\le|G/H|\le\kappa$. Then the family $$\U=\big\{\bigcap_{i=1}^nx_iH_ix_i^{-1}:H_1,\dots,H_n\in\HH,\;x_1,\dots,x_n\in G\big\}$$ has cardinality $|\U|\le\kappa$ and consists of subgroups of index $\le\kappa$ in $G$.
Now consider the  topology $\tau$ on $G$ consisting of sets $U\subset G$ such that for every point $x\in U$ there is a subgroup $H\in\U$ such that $xH\subset U$. Using Theorem~1.3.2 in \cite{AT}, it can be shown that the topology $\tau$ turns $G$ into a linear topological group of weight $\le \kappa$. Then $w_l(G)\le\kappa=w_l'(G)$.
\end{proof}

\begin{proposition}\label{p:wl<wgw} Each infinite group $G$ has $w_g(G)\le w_l(G)\le w_g(G)^\w$.
\end{proposition}

\begin{proof} The inequality $w_g(G)\le w_l(G)$ is trivial. To prove that $w_l(G)\le w_g(G)^\w$, fix a Hausdorff group topology $\tau$ on $G$ such that $w(G,\tau)=w_g(G)$. Let $\kappa=w_g(G)=w(G,\tau)$ and fix a neighborhood base $\{U_\alpha\}_{\alpha\in \kappa}\subset\tau$ at the unit $1_G$ of the group $G$. For every $\alpha\in\kappa$ put $U_{\alpha,0}=U_{\alpha}$ and for every $n\in\IN$ choose a symmetric neighborhood $U_{\alpha,n}=U_{\alpha,n}^{-1}\in\tau$ of the unit $1_G$ such that $U_{\alpha,n}U_{\alpha,n}\subset U_{\alpha,n-1}$. It is easy to see that the intersection $H_{\alpha}=\bigcap_{n\in\IN}U_{\alpha,n}$ is a subgroup of $G$. We claim that this subgroup has index $|G/H_\alpha|\le \kappa^\w$ in $G$. Let $q:G\to G/H_{\alpha}:=\{xH_\alpha:x\in G\}$, $q:x\mapsto xH_\alpha$, be the quotient map and $s:G/H_{\alpha}\to G$ be any function such that $q\circ s=\id$. For every $n\in\IN$ choose a maximal subset $K_n\subset G$ such that $xU_{\alpha,n}\cap yU_{\alpha,n}=\emptyset$ for any distinct points $x,y\in K_n$, and observe that $K_n$ is a discrete subspace of $(G,\tau)$, which implies that $|K_n|\le w(G,\tau)=w_g(G)=\kappa$. By the maximality of $K_n$, for every $x\in G$ there is a point $f_n(x)\in K_n$ such that $xU_{\alpha,n}\cap f_n(x)U_{\alpha,n}\ne\emptyset$ and hence $x\in f_n(x)U_{\alpha,n}U_{\alpha,n}^{-1}\subset f_n(x)U_{\alpha,n-1}$. The functions $f_n$, $n\in\IN$, form the function $f=(f_n)_{n=1}^\infty:G\to\prod_{n=1}^\infty K_n$. We claim that the composition $f\circ s:G/H_{\alpha}\to \prod_{n\in\IN}K_n$ is injective. Given any cosets $X,Y\in G/H_\alpha$ with $f\circ s(X)=f\circ s(Y)$, consider the points $x=s(X)$ and $y=s(Y)$. The equality $f(x)=f\circ s(X)=f\circ s(Y)=f(y)$ implies that $f_n(x)=f_n(y)$ for all $n\in\IN$ and hence $x,y\in f_n(x)U_{\alpha,n-1}$, which implies that $x^{-1}y\in U^{-1}_{\alpha,n-1}U_{\alpha,n-1}\subset U_{\alpha,n-2}$ for every $n\ge 2$. Then $x^{-1}y\in\bigcap_{n=2}^\infty U_{\alpha,n-2}=H_\alpha$ and hence $X=xH_\alpha=yH_\alpha=Y$, witnessing that the function $f\circ s:G/V_{\alpha}\to\prod_{n\in\IN}K_\alpha$ is injective and hence
$|G/V_{\alpha}|\le\big|\prod_{n\in\IN}K_\alpha\big|\le\kappa^\w$.
Since $1_G\in\bigcap_{\alpha\in\kappa}H_\alpha\subset\bigcap_{\alpha\in\kappa}U_\alpha=\{1_G\}$, we can apply Proposition~\ref{p:lw-char} and conclude that $w_l(G)\le\kappa^\w=w_g(G)^\w$.
\end{proof}

Modifying the proofs of Propositions~\ref{p:lw-char} and \ref{p:wl<wgw} we can prove the following two facts about the cardinal group invariant $u_l$.

\begin{proposition}\label{p:ul-char} For an infinite group $G$ its linear uniform spread $u_l(G)$ is equal to the smallest cardinal $\kappa$ such that $\{1_G\}=\bigcap\mathcal H$ for some family $\mathcal H$ of subgroups of index $\le \kappa$ in $G$.
\end{proposition}

\begin{proposition}\label{p:ul<ugw} Each infinite group $G$ has $u_g(G)\le u_l(G)\le u_g(G)^\w$.
\end{proposition}

Both inequalities in this Propositions~\ref{p:wl<wgw} and \ref{p:ul<ugw} can be strict.

\begin{example} The discrete group $G=\mathbb Z$ of integers has $u_g(G)=u_l(G)=w_g(G)=w_l(G)=\w<\w^\w$.
\end{example}

\begin{example} Let $M$ be the unit interval or the unit circle. The homeomorphism group $G$ of $M$ has
$u_g(G)=w_g(G)=\w$ and $u_l(G)=w_l(G)=\mathfrak c$.
\end{example}

\begin{proof} The inequality $w_g(G)=\w$ follows from the fact that the compact-open topology on the homeomorphism group $G$ is metrizable and separable. The equality $u_l(G)=\mathfrak c$ follows from Proposition~\ref{p:ul-char} and Theorem 6 of \cite{RS} saying that the subgroup $G_+\subset G$ of orientation preserving homeomorphisms of $M$ is the only subgroup of index $<\mathfrak c$ in $G$.
\end{proof}

The following proposition gives an upper bound on the weak compatibility number $\cn$.

\begin{proposition}\label{p:wc<lw} Each infinite group $G$ has $\cn(G)\le w_l(G)$.
\end{proposition}

\begin{proof} It suffices to prove that $G$ satisfies the weak $\kappa^+$-compatibility condition for $\kappa=w_l(G)$. By Proposition~\ref{p:lw-char}, there is a family $\mathcal H$ of subgroups of index $\le\kappa$ in $G$ such that $|\mathcal H|\le\kappa$ and $\{1_G\}=\bigcap\HH$. Replacing $\HH$ by a larger family, we can assume that for any subgroups $H_1,\dots,H_n\in\HH$ and any points $x_1,\dots,x_n\in G$ the subgroup $\bigcap_{i=1}^nx_iH_ix_i^{-1}$ belongs to the family $\HH$.

To show that the group $G$ satisfies the weak $\kappa^+$-compatibility condition, fix a finite group $F$ and  isomorphisms $f_i:F\to F_i$, $i<\kappa^+$, onto finite subgroups $F_i\subset G$. Since the family $\mathcal H$ is closed under finite intersections, for every
 $i<\kappa^+$ we can choose a subgroup $H_i\in\mathcal H$ such that $F_i\cap H_i=\{1_G\}$. Replacing $H_i$ by the subgroup $\bigcap_{x\in F_i}xH_ix^{-1}$, we can assume that $xH_ix^{-1}=H_i$ for all points $x\in F_i$. Since $|\mathcal H|\le\kappa<\kappa^+$, for some subgroup $H\in\HH$ the set $I=\{i\in\kappa^+:H_i=H\}$ has cardinality $|I|=\kappa^+$. Consider the family of left cosets $G/H=\{xH:x\in G\}$ and the quotient map $q:G\to G/H$, $q:x\mapsto xH$. Since $|(G/H)^F|\le\kappa<\kappa^+$, there are two indices $i<j$ in $I$ such that $q\circ f_i=q\circ f_j$. Now consider the subgroup $F_{ij}=\langle F_i\cup F_j\rangle$ generated by the set $F_i\cup F_j$. Taking into account that $xHx^{-1}=H$ for all $x\in F_i\cup F_j$, we conclude that $xHx^{-1}=H$ for all $x\in F_{ij}$, which implies that $L=F_{ij}\cdot H=\{xy:x\in F_{ij},\;y\in H\}$ is a subgroup of $G$ and $H$ is a normal subgroup in $L$. Consequently, the subspace $L/H=\{xH:x\in L\}$ has the structure of a group and the restriction $q|L:L\to L/H$ is a group homomorhism. The equality $q\circ f_i=q\circ f_j$ implies that $L/H=q(F_{ij})=q(F_i)=q(F_j)$. Since $H\cap F_i=\{1_G\}$, the restriction $\varphi=q|F_i:F_i\to L/H$, being injective and surjective, is an isomorphism.
Then $\psi=f_i^{-1}\circ\varphi^{-1}:L/H\to F$ is an isomorphism too. It can be shown that the homomorphism $\phi=\psi\circ q|F_{ij}:F_{ij}\to F$ has the required property: $\phi\circ f_j=\phi\circ f_i=\id_F$.
\end{proof}

\begin{question} Is $\cn(G)\le w_g(G)$ for any group $G$?
\end{question}

Now we are able to prove two equalities of Theorem~\ref{main}.

\begin{lemma}\label{l1} For every infinite cardinal $\kappa$ and group $G$ with $\Alt(\kappa)\subset G\subset \Sym(\kappa)$ we get $$ \kappa=\cn(G)=w_l(G).$$
\end{lemma}

\begin{proof} By Proposition~\ref{p:wc<lw}, $\cn(G)\le w_l(G)$. Since the cardinal group invariants $\cn$ and $w_l$ are monotone, it suffices to prove that $w_l(\Sym(\kappa))\le\kappa$ and $\cn(\Alt(\kappa))\ge\kappa$.

To see that $w_l(\Sym(\kappa))\le\kappa$, for every $i\in \kappa$ consider the subgroup $G_i=\{f\in\Sym(\kappa):f(i)=i\}$ and observe that the index of this subgroup in $\Sym(\kappa)$ is equal to $\kappa$. Taking into account that
$\bigcap_{i\in\kappa}G_i=\{\id\}$ is the trivial subgroup of $\Sym(\kappa)$,  and applying Proposition~\ref{p:lw-char}, we conclude that $w_l(\Sym(\kappa))\le\kappa$.

To see that $\cn(\Alt(\kappa))\ge\kappa$, it suffices to check that for every $\lambda<\kappa$ the group $\Alt(\kappa)$ does not satisfy the weak $\lambda^+$-compatibility condition.
For every ordinal $3\le i<\lambda^+\le\kappa$ consider the 4-element subset $K_i=\{0,1,2,i\}$ and its alternating group $\Alt(K_i)\subset \Alt(\kappa)$. Fix an isomorphism $f_i:\Alt(4)\to \Alt(K_i)$ and observe that for any $3\le i<j<\lambda^+$ the subgroup $\langle \Alt(K_i)\cup\Alt(K_j)\rangle$ is equal to $\Alt(K_i\cup K_j)$ and is isomorphic to the alternating group $\Alt(5)$.
Since $\Alt(5)$ is a simple group and $\Alt(4)$ is not
simple there does not exist
a surjective homomorphism from $\langle \Alt(H_i)\cup \Alt(H_j) \rangle$ to $\Alt(4)$, which implies that the group $\Alt(\kappa)$ does not satisfy the weak $\lambda^+$-compatibility condition.
\end{proof}

 We say that a topological space $X$ is {\em $\sigma$-discrete} if $X$ can be written as a countable union
of discrete subspaces. By \cite{BG} or \cite[6.1]{BGP}, for every cardinal $\kappa$ the group $\Sym_\fin(\kappa)$ is $\sigma$-discrete in each shift-invariant Hausdorff topology on $\Sym_\fin(\kappa)$. The same fact is true for the alternating group  $\Alt(\kappa)$.

\begin{theorem}\label{sigmad} The alternating group $\Alt(\kappa)$ on a cardinal $\kappa$ is $\sigma$-discrete in any Hausdorff shift-invariant topology $\tau$ on $\Alt(\kappa)$.
\end{theorem}

\begin{proof} If the cardinal $\kappa$ is finite, then the alternating group $\Alt(\kappa)$ is finite and hence discrete in the topology $\tau$. So, we assume that $\kappa$ is infinite.  To prove the proposition, it suffices to check that for every $n\in\w$ the subspace $\Alt_{n}(\kappa)=\{f\in\Alt(\kappa):|\supp(f)|= n\}$ of $\Alt(\kappa)$ is discrete. Given any permutation $f\in \Alt_{n}(\kappa)$ we shall construct an open set $O_f\in\tau$ such that $f\in O_f\cap \Alt_n(\kappa)\subset \{g\in \Alt(\kappa):\supp(g)=\supp(f)\}$.

Choose two disjoint subsets $A,B\subset \kappa\setminus\supp(f)$ of cardinality $|A|=|B|=n+1$ and for any points $x\in\supp(f)$ and $a\in A$, $b\in B$  consider the even permutation $\pi_{x,a,b}\in \Alt_{3}(\kappa)$ with support $\supp(\pi_{x,a,b})=\{x,a,b\}$ such that $\pi_{x,a,b}(x)=a$, $\pi_{x,a,b}(a)=b$ and $\pi_{x,a,b}(b)=x$. Since the topology $\tau$ on $\Alt(\kappa)$ is shift-invariant and Hausdorff, the set $O_{x,a,b}=\{g\in\Alt(\kappa):x_{x,a,b}\circ g\ne g\circ \pi_{x,a,b}\}$ is $\tau$-open. Taking into account that $\pi_{x,a,b}\circ f(x)=f(x)\ne f(a)=f\circ \pi_{x,a,b}(x)$, we conclude that the permutation $f$ belongs to the $\tau$-open set $O_{x,a,b}$ and $O_f=\bigcap_{x\in \supp(f)}\bigcap_{a\in A}\bigcap_{b\in B}O_{x,a,b}$ is a $\tau$-open neighborhood of $f$. We claim that the open set $O_f$ has the desired property: $O_f\cap\Alt_{n}(\kappa)\subset\{g\in \Alt(\kappa):\supp(g)=\supp(f)\}$. Take any permutation $g\in O_f\cap\Alt_n(\kappa)$. Assuming that $\supp(g)\ne\supp(f)$ and taking into account that $|\supp(g)|=|\supp(g)|=n$, we conclude that $\supp(f)\setminus\supp(g)$ contains some point $x$. Since $|A|=|B|>|\supp(g)|$, we can choose points $a\in A\setminus\supp(g)$ and $b\in B\setminus\supp(g)$. Then $\supp(\pi_{x,a,b})\cap \supp(g)=\{x,a,b\}\cap\supp(g)=\emptyset$ and hence $\pi_{x,a,b}\circ g=g\circ \pi_{x,a,b}$, which contradicts the inclusion $g\in O_f$. This contradiction shows that the subspace $O_f\cap\Alt_n(\kappa)\subset\{g\in \Alt(\kappa):\supp(g)=\supp(f)\}$ is finite and hence discrete. Then the point $f$ is isolated in the discrete $\tau$-open subset $O_f\cap\Alt_n(\kappa)$ of $\Alt_n(\kappa)$ and hence is isolated in $\Alt_n(\kappa)$, witnessing that the subspace $\Alt_n(\kappa)$ of the topological space $(\Alt(\kappa),\tau)$ is discrete.
\end{proof}

In the following lemma we prove another equality of Theorem~\ref{main}.

\begin{lemma}\label{l2} For any infinite cardinal $\kappa$ and any group $G$ with $\Alt(\kappa)\subset G\subset\Sym(\kappa)$ we get
$\kappa=s_s(G)$.
\end{lemma}

\begin{proof} Lemma~\ref{l1} and obvious inequalities between the cardinal group invariants imply $s_s(G)\le w_l(G)=\kappa$. It remains to prove that $s_s(G)\ge \kappa$. Assuming that $s_s(G)<\kappa$ we could find a Hausdorff shift-invariant topology $\tau$ with spread $s(G,\tau)<\kappa$. By Proposition~\ref{sigmad},  the subgroup $\Alt(\kappa)$ of $G$ is $\sigma$-discrete in the topology $\tau$. Consequently, $\Alt(\kappa)=\bigcup_{i\in\w}D_i$ where each subspace $D_i$ is discrete in $(G,\tau)$. Since $|\Alt(\kappa)|=\kappa>s(G,\tau)$, some set $D_i$ has cardinality $|D_i|>s(G,\tau)$, which contradicts the definition of the spread $s(G,\tau)$. This contradiction shows that $s_s(G)\ge\kappa$.
\end{proof}

Estabilishing the equality $\kappa=u_c(G)$ in Theorem~\ref{main} is the most difficult part of the proof, which requires some preparatory work.

For a cardinal $\kappa$ and a subgroup $G\subset\Sym(\kappa)$ by $\tau_p$ we denote the topology of point-wise convergence on $G$. This topology is generated by the subbase consisting of the sets
$$G(a,b)=\{g\in G:g(a)=b\}\mbox{ \ where \ }a,b\in \kappa.$$ By Theorem~2.1 of \cite{BGP}, on any group $G$ with $\Sym_\fin(\kappa)\subset G\subset\Sym(\kappa)$, the topology $\tau_p$ coincides with the {\em restricted Zariski topology} $\Zeta_G'$ on $G$, generated by the subbase consisting of the sets $$G\setminus\{a\},\;\;\{x\in G:xbx^{-1}\ne aba^{-1}\}\mbox{ \  and \ }\{x\in G:(xcx^{-1})b(xcx^{-1})^{-1}\ne b\}$$ where $a,b,c\in G$ and $b^2=c^2=1_G$. It is easy to check that the topology $\Zeta'_G$ is shift-invariant. The following theorem generalizes Theorem~2.1 of \cite{BGP}.

\begin{theorem}\label{t:zariski} Let $\kappa$ be a cardinal. For any group $G$ with $\Alt(\kappa)\subset G\subset\Sym(\kappa)$ the topology $\tau_p$ of pointwise convergence on $G$ coincides with the restricted Zariski topology $\Zeta_G'$.
\end{theorem}

\begin{proof} The proof of this theorem is just a suitable modification of the proof of Theorem 2.1 \cite{BGP}. Fix a cardinal $\kappa$ and a subgroup $G\subset\Sym(\kappa)$ containing the alternating group $\Alt(\kappa)$. If the cardinal $\kappa$ is finite, then so is the group $G$. In this case the $T_1$-topologies $\tau_p$ and $\Zeta'_G$ coincide with the discrete topology on $G$. So, we assume that the cardinal $\kappa$ is infinite.

For two distinct elements $x,y\in \kappa$ consider the unique transposition $t_{x,y}$ with support $\supp(t_{x,y})=\{x,y\}$. The transposition $t_{x,y}$ exchanges $x$ and $y$ by their places and does not move other points of $\kappa$. For a subset $A\subset\kappa$ consider the subgroups $G(A)=\{g\in G:\supp(g)\subset A\}$ and $G_A=\{g\in G:\supp(g)\cap A=\emptyset\}=\{g\in G:g|A=\id|A\}$.

\begin{lemma} For any 6-element subset $A\subset \kappa$ the subgroup $G_A$ is $\Zeta'_G$-closed in $G$.
\end{lemma}

\begin{proof} Given any permutation $f\in G\setminus G_A$ find a point $a\in A$ with $f(a)\ne a$. Next, choose a point $b\in A\setminus \{a,f(a)\}$ and two distinct points $c,d\in A\setminus\{a,b,f(a),f(b)\}$. Consider the even permutation $t=t_{a,b}\circ t_{c,d}\in\Alt(\kappa)\subset G$ and observe that $t\circ f(a)=f(a)\ne f(b)=f\circ t(a)$. Since $\supp(t)=\{a,b,c,d\}\subset A$, the transposition $t$ commutes with all permutations $g\in G_A$, which implies that $$O_f=\{g\in G:t\circ g\ne g\circ t\}=\{g\in G:tgt^{-1}\ne g\}$$is a $\Zeta'_G$-open neighborhood of $f$, disjoint with the subgroup $G_A$.
\end{proof}

\begin{lemma}\label{l:open} For each 6-element subset $A\subset \kappa$ the subgroup $G_A$ is $\Zeta_G'$-open in $G$.
\end{lemma}

\begin{proof} To derive a contradiction, assume that for some 6-element set $A'\subset \kappa$ the subgroup $G_{A'}$ is not $\Zeta'_G$-open. Being $\Zeta'_G$-closed, this subgroup is nowhere dense in the semi-topological group $(G,\Zeta'_G)$. Observe that for any 6-element subset $A\subset \kappa$ and any even permutation $f\in \Alt(G)$ with $f(A)=A'$, we get $G_{A}=f^{-1}\circ G_{A'}\circ f$, which implies that the subgroup $G_A$ is closed and nowhere dense in $(G,\Zeta'_G)$.

We claim that for every 6-element set $A\subset \kappa$ and any finite subset $B\subset \kappa$ the set $G(A,B)=\{g\in G:g(A)\subset B\}$ is nowhere dense in $(G,\Zeta'_G)$. Since $A$ and $B$ are finite, we can choose a finite set $F\subset G$ such that for any $g\in G(A,B)$ there exists $f\in F$ with $g|A=f|A$. Then $f^{-1}\circ g|A=\id_A$ which implies that $G(A,B)=\bigcup_{f\in F}f\circ G_A$ is nowhere dense in $(G,\Zeta'_G)$.

Now fix any four pairwise disjoint 6-element subsets $A_1,A_2,A_3,A_4\subset \kappa$ and consider their union $A=\bigcup_{i=1}^4A_i$. Consider the finite subset $T=\{t_{a_1,a_2}\circ t_{a_3,a_4}:(a_i)_{i=1}^4\in \prod_{i=1}^4A_i\}\subset \Alt(\kappa)\subset G$. For any permutations $t,s\in T$ with $t\circ s\ne s\circ t$ the set
$$U_{t,s}=\{u\in G:(usu^{-1})t(usu^{-1})^{-1}\ne t\}$$is a $\Zeta'_G$-open neighborhood of $1_G$ by the definition of the topology $\Zeta'_G$. Since $T$ is finite, the intersection
$$U=\bigcap\{U_{t,s}:t,s\in T,\;ts\ne st\}$$is a $\Zeta'_G$-open neighborhood of $1_G$. Choose a permutation $u\in U$ which does not belong to the nowhere dense subset $\bigcup_{i=1}^4G(A_i,A)$. For every $i\in\{1,2,3,4\}$ there is a point $a_i\in A_i$ such that $u(a_i)\notin A$. Choose any point $a_2'\in A_2\setminus\{a_2\}$ and consider the non-commuting permutations $t=t_{a_1,a_2'}\circ t_{a_3,a_4}$ and $s=t_{a_1,a_2}\circ t_{a_3,a_4}$ in $T$.
It follows from $u\in U$ that the permutation $v=usu^{-1}$ does not commute with the permutation $t$. On the other hand, the support $\supp(v)=\supp(usu^{-1})=u(\{a_1,a_2,a_3,a_4\})$ does not intersect the set $A\supset \{a_1,a_2',a_3,a_4\}=\supp(t)$, which implies that $vt=tv$. This contradiction completes the proof of Lemma~\ref{l:open}.
\end{proof}

Now we able to prove that $\Zeta'_G=\tau_p$. Taking into account that $\tau_p$ is a Hausdorff group topology on $G$, we conclude that $\Zeta'_G\subset\tau_p$. To prove the reverse inclusion, it suffices to check that for every $a,b\in \kappa$ the subbasic $\tau_p$-open set $G(a,b)=\{g\in G:g(a)=b\}$ is $\Zeta'_G$-open. Choose any 6-element subset $A\subset\kappa$ containing the point $a$. By Lemma~\ref{l:open} the subgroup $G_A$ is $\Zeta'_G$-open and hence for any $g\in G(a,b)$ the set $g\circ G_A\subset G(a,b)$ is a $\Zeta'_G$-open neighborhood of $g$, witnessing that the set $G(a,b)$ is $\Zeta'_G$-open.
\end{proof}

Combining Theorem~\ref{t:zariski} with Proposition~4.1 of \cite{BGP}, we obtain the following generalization of Theorem 4.2 of \cite{BGP}.

\begin{theorem}\label{t:tp} Let $\kappa$ be an infinite cardinal and $G$ be a group such that $\Alt(\kappa)\subset G\subset\Sym(\kappa)$. Each shift-invariant $T_1$-topology $\tau$ with separately continuous commutator on $G$ contains the topology $\tau_p$ of pointwise convergence on $G$.
\end{theorem}

Now we are able to prove the final equality of Theorem~\ref{main}.

\begin{lemma}\label{l3} For any infinite cardinal $\kappa$ and any group $G$ with $\Alt(G)\subset G\subset\Sym(G)$ we get $\kappa=u_c(G)$.
\end{lemma}

\begin{proof} It follows from Lemma~\ref{l1} that $u_c(G)\le w_l(G)=\kappa$. So, it remains to prove that $u_c(G)\ge \kappa$. To derive a contradiction, assume that $u_c(G)<\kappa$ and find a  topology $\tau\in\Tau_c(G)$ such that $u(G,\tau)<\kappa$. By Theorem~\ref{t:tp}, $\tau_p\subset \tau$, which implies that the subgroup $G_0=\{g\in G:g(0)=0\}$ is $\tau$-open. For every $i\in\kappa\setminus \{0\}$ fix a permutation $g_i\in\Alt(\kappa)$ such that $g_i(0)=i$ and observe that for every $0<i\ne j<\kappa$ we get $g_j\notin g_i\circ G_0$, which means that the set $D=\{g_i\}_{0<i<\kappa}$ is uniformly discrete in $(G,\tau)$. So, $u(G,\tau)\ge \kappa$, which is a desired contradiction.
\end{proof}

\begin{problem} Let $\kappa$ be a cardinal and $G$ be a group with $\Alt(\kappa)\subset G\subset\Sym(\kappa)$. Calculate the value of the cardinal invariant $u_s(G)$. Is $u_s(G)=\kappa$?
\end{problem}

The equalities in Theorem~\ref{main} hold also for some other groups, in particular for the additive group $\IR$ of real numbers.

\begin{example} The group $\IR$ does admit a linear group topology with countable weight and hence has $u_s(\IR)=w_l(\IR)=\cn(\IR)=\w$.
\end{example}

The equalities in this example follows from our next theorem evaluating the cardinal group invariants of infinite abelian groups. 

\begin{theorem}\label{t:ab} Each infinite abelian group $G$ has $$
\log\log |G|\le s_s(G)\le w_s(G)=w_l(G)=\log|G|\mbox{ \ and \ }
u_s(G)=u_l(G)=\cn(G)=\w.
$$
\end{theorem}

\begin{proof} The inequalities $s_s(G)\le w_s(G)\le w_l(G)$ hold for any infinite group $G$. The lower bounds $\log\log|G|\le s(G)$ and $\log|G|\le w(G)$ follow from the inequalities $|X|\le 2^{2^{s(X,\tau)}}$ and $|X|\le 2^{w(X,\tau)}$ holding for any Hausdorff topological space $(X,\tau)$ (see Theorems 5.5 and 3.1 in \cite{Hodel}).
\smallskip

Now we prove that $w_l(G)\le\log|G|$. Since each infinite abelian group embeds into a divisible abelian group of the same cardinality \cite[4.1.6]{Rob}, we can assume that $G$ is divisible (which means that for every $x\in G$ and $n\in\IN$ there exists $y\in G$ such that $y^n=x$). It is clear that the additive group $\IQ_0$ of rational numbers is divisible and so is the quasi-cyclic $p$-group $\IQ_p=\{z\in\IC:\exists k\in\IN$ $z^{p^k}=1\}$ for any prime number $p$. Let $\IP$ denote the set of all prime numbers and $\IP_0=\{0\}\cup\IP$. Endow the countable groups $\IQ_p$, $p\in\IP_0$, with the discrete topologies and for the cardinal $\kappa=\log|G|$ consider the Tychonoff product $L=\prod_{p\in\IP_0}\IQ_p^\kappa$. Taking into account that $L$ is a linear topological group of weight $\kappa$, we conclude that $w_l(L)\le \kappa$. It remains to prove that the divisible group $G$ embeds into $L$.

For a cardinal $\lambda$ and an abelian group $A$ by $A^{(\lambda)}$ we denote the direct sum of $\lambda$ many copies of the groups $A$. By \cite[4.1.5]{Rob}, the group $G$, being divisible, is isomorphic ($\approx$) to the direct sum $\bigoplus_{p\in \IP_0}\IQ_p^{(k_p)}$ for some cardinals $k_p\le|G|$, $p\in\IP_0$. By the same reason, for every $p\in\IP_0$ the $\kappa$-th power $\IQ_p^\kappa$ of $\IQ_p$ is isomorphic to the direct sum $(\IQ_0\oplus\IQ_p)^{(2^\kappa)}$ of $2^\kappa$ many copies of the groups $\IQ_0\oplus \IQ_p$. Since $k_p\le|G|\le 2^\kappa$ for all $p\in\IP_0$, the group $G\approx   \bigoplus_{p\in \IP_0}\IQ_p^{(k_p)}$ embeds into the direct sum $\bigoplus_{p\in\IP_0}\IQ_p^{(2^\kappa)}\approx L$. Now the monotonicity of the cardinal group invariant $w_l$ ensures that $w_l(G)\le w_l(L)\le\kappa=\log|G|$. Combining this inequality with $w_s(G)\ge\log |G|$, we get the equalities $w_s(G)=w_l(G)=\log |G|$.
\smallskip

Next, we prove that $u_s(G)=u_l(G)=\w$. Since $\w\le u_s(G)\le u_l(G)$, it suffices to show that $u_l(G)\le \w$. Since the cardinal group invariant $u_l$ is monotone, we can assume that the group $G$ is divisible and hence is isomorphic to the direct sum of countable abelian groups. This implies that the trivial subgroup of $G$ can be written as the intersection of subgroups of at most countable index in $G$. By Proposition~\ref{p:ul-char}, $u_l(G)\le \w$.
\smallskip

Finally, we prove that $\cn(G)=\w$. It suffices to prove that $G$ satisfies the weak $\w^+$-compatibility condition. Fix a finite group $F$ and isomorphisms $f_i:F\to F_i$, $i\in\w_1$, onto subgroups $F_i\subset G$. The $\Delta$-Lemma \cite[9.18]{Jech} yields an uncountable subset $\Omega\subset\w_1$ and a finite subgroup $D\subset G$ such that $F_i\cap F_j=D$ for any distinct indices $i,j\in\Omega$. By the Pigeonhole Principle, for some subgroup $E\subset F$ the set $\Omega_1=\{i\in\Omega: f^{-1}_i(D)=E\}$ is uncountable. Since the set of isomorphisms from $E$ to $D$ is finite, for some isomorphism $f:E\to D$ the set $\Omega_2=\{i\in\Omega_1:f_i|E=f\}$ is uncountable. Now take any two ordinals $i<j$ in the set $\Omega_2$ and consider the subgroup $F_{ij}=F_i+F_j$ of $G$. Define a homomorphism $\phi:F_{ij}\to F$ assigning to a point $x=x_i+x_j$ with $x_i\in F_i$ and $x_j\in F_j$ the point $\phi(x)=f_i^{-1}(x_i)+f_j^{-1}(x_j)$. Let us show that the map $\phi$ is well-defined. Assume that $x=x_i'+x_j'$ for some points $x_i'\in F_i$ and $x_j'\in F_j$. Then $0=x-x=(x_i-x_i')+(x_j-x_j')$ implies that $x_i-x_i'=x_j'-x_j\in F_i\cap F_j=D$ and hence
$$f_i^{-1}(x_i)-f_i^{-1}(x_i')=f_i^{-1}(x_i-x_i')=f^{-1}(x_i-x_i')=f^{-1}(x_j'-x_j)=f_j^{-1}(x_j'-x_j)=f_j^{-1}(x_j')-f_j^{-1}(x_j)$$and finally $f_i^{-1}(x_i)+f_j^{-1}(x_j)=f_i^{-1}(x_i')+f_j^{-1}(x_j')$. Therefore the map $\phi:F_{ij}\to F$ is a well-defined homomorphism such that $\phi\circ f_i=\phi\circ f_j=\id_F$, witnessing that the group $G$ satisfies the weak $\w^+$-compatibility condition and $wc(G)=\w$.
\end{proof}

Theorem~\ref{t:ab} implies that for infinite abelian groups the diagram describing the relations between the cardinal group invariants takes the following form:
$$
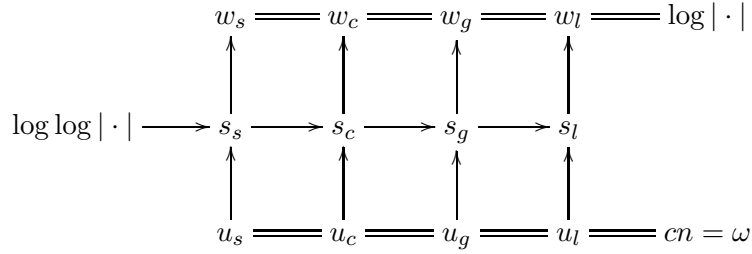
\xymatrix{
&w_s\ar@{=}[r]&w_c\ar@{=}[r]&w_g\ar@{=}[r]&w_l\ar@{=}[r]&\log|\cdot|\\
\log\log|\cdot|\ar[r]&s_s\ar[r]\ar[u]&s_c\ar[r]\ar[u]&s_g\ar[r]\ar[u]&s_l\ar[u]\\
&u_s\ar@{=}[r]\ar[u]&u_c\ar@{=}[r]\ar[u]&u_g\ar@{=}[r]\ar[u]&u_l\ar[u]&\cn=\w\ar@{=}[l]
}
$$
Looking at this diagram it is natural to ask about the values of the cardinal invariants in the middle row. Are they equal to $\log\log|G|$ or to $\log|G|$? Surprisingly, the answer depends on additional set-theoretic assumptions.

\begin{theorem} \begin{enumerate}
\item PFA implies that each group $G$ of cardinality $|G|>\mathfrak c$ has $s_s(G)>\w=\log\log \mathfrak c^+$.
\item For any regular uncountable cardinals $\lambda\le\kappa$ it is consistent that $2^\w=\lambda$, $2^{\w_1}=\kappa$ and each infinite abelian group $G$ of cardinality $|G|\le \kappa$ has linear spread $s_l(G)=\w=\log\log|G|$.
\end{enumerate}
\end{theorem}

\begin{proof} 1. The first statement follows from a result of Todorcevic \cite[8.12]{Tod} saying that under PFA each Hausdorff space $X$ of countable spread has cardinality $|X|\le\mathfrak c$.
\smallskip

2. By \cite[4.10]{Juh02} for any regular uncountable cardinals $\lambda\le\kappa$ it is consistent that $2^\w=\lambda$, $2^{\w_1}=\kappa$ and there exists a subspace $X\subset \{0,1\}^\lambda$ of cardinality $|X|=\kappa$ such that each finite power $X^n$ is hereditarily separable and hence has countable spread. The following lemma completes the proof the statement (2).
\end{proof}

\begin{lemma} Let $X$ be a zero-dimensional space such that each finite power $X^n$ of $X$ has countable spread. Then any infinite abelian group $G$ of cardinality $|G|\le |X|$ has linear spread $s_l(G)=\w$.
\end{lemma}

\begin{proof} Consider the space $C_p(X,\IQ)\subset\IQ^X$ of continuous functions from $X$ to the discrete group $\IQ$ of rational numbers, and let $C_pC_p(X,\IQ)$ be the space of continuous functions from $C_p(X,\IQ)$ to $\IQ$. Observe that the topology on $C_pC_p(X,\IQ)$ inherited from the Tychonoff product $\IQ^{C_p(X,\IQ)}$ is linear.

The zero-dimensionality of $X$ can be used to prove that the map $\delta:X\to C_pC_p(X,\IQ)$ assigning to each $x\in X$ the Dirac measure $\delta_x:C_p(X,\IQ)\to\IQ$, $\delta_x:f\mapsto f(x)$, is a topological embedding. Let $L$ be the $\IQ$-linear hull of the set $\delta(X)$ in the $\IQ$-linear space $C_pC_p(X,\IQ)$ and $Z$ be the group hull of $\delta(X)$ in $C_pC_p(X,\IQ)$. It can be shown that $Z$ is a closed subgroup of $L$, so we can consider the quotient group $L/Z$ and notice that it is a divisible torsion group isomorphic to the direct sum $\bigoplus_{p\in\IP}\IQ_p^{(|X|)}$. Here $\IP$ is the set of prime numbers and $\IQ_p=\{z\in\IC:\exists k\in\IN \;z^{p^k}=1\}$ for $p\in\IP$. Consequently, the  group $L\times (L/Z)$ is isomorphic to $\bigoplus_{p\in\IP_0}\IQ_p^{(|X|)}$ where $\IP_0=\IP\cup\{0\}$ and $\IQ_0=\IQ$. It can be shown that the space $L\times (L/Z)$ can be written as the countable union of spaces homeomorphic to continuous images of finite powers of $X$, which implies that the space $L\times(L/Z)$ has countable spread.

Since the group $L\times (L/Z)$ carries a linear Hausdorff group topology of countable spread, so does the group $\bigoplus_{p\in\IP_0}\IQ_p^{(|X|)}$. By \cite[4.1.5 and 4.1.6]{Rob}, every infinite abelian group $G$ of cardinality $|G|\le |X|$ embeds into $\bigoplus_{p\in\IP_0}\IQ_p^{(|G|)}\subset\bigoplus_{p\in\IP_0}\IQ_p^{(|X|)}$ and hence $G$ carries a linear Hausdorff group topology with countable spread, which implies that $s_l(G)=\w$.
\end{proof}



{\bf Acknowledgement.} We thank Siegfried B\"ocherer for bringing the embeddability question to our attention and we thank Istvan Juh\'asz for his hint to
\cite{Juh02}.


\newpage

\begin{thebibliography}{MM}



\bibitem{AT} A.~Arhangel'skii, M.~Tkachenko, {\em  Topological groups and related structures}, Atlantis Press, Paris; World Scientific Publishing Co. Pte. Ltd., Hackensack, NJ, 2008.

\bibitem{BG} T.Banakh, I.Guran, {\em Perfectly supportable semigroups are $\sigma$-discrete in each Hausdorff shift-invariant topology}, Topological Algebra and Applications, {\bf 1} (2013) 1--8.

\bibitem{BGP} T.Banakh, I.Guran, I.Protasov, {\em Algebraically determined topologies on permutation groups}, Topology Appl. {\bf 159}:9 (2012) 2258--2268

\bibitem{Hodel} R.~Hodel, {\em Cardinal functions. I}, Handbook of set-theoretic topology, 1--61, North-Holland, Amsterdam, 1984.


\bibitem{Jech} T.~Jech, {\em Set Theory. The third millennium edition, revised and expanded}, Springer Monographs in Mathematics. Springer-Verlag, Berlin, 2003.


\bibitem{Juh02} I.~Juh\'asz, {\em HFD and HFC type spaces, with applications}, Topology Appl. {\bf 126}:1-2 (2002), 217--262.

\bibitem{JST} W.~Just, S.~Shelah, S.~Thomas, {\em The automorphism tower problem revisited}, Adv. Math. {\bf 148}:2 (1999), 243--265.

\bibitem{Rob} D.~Robinson, {\em A course in the theory of groups}, Springer-Verlag, New York-Berlin, 1982.

\bibitem{RS} C.~Rosendal, S.~Solecki, {\em Automatic continuity of homomorphisms and fixed points on metric compacta}, Israel J. Math. {\bf 162} (2007), 349--371.

\bibitem{Tod} S.~Todorcevic, {\em Partition problems in topology}, Contemporary Mathematics, 84. Amer.~Math.~Soc., Providence, RI, 1989.

\end{thebibliography}
\end{document}